\documentclass[a4paper,reqno,11pt, oneside]{amsart}

\usepackage{amssymb}
\usepackage{amstext}
\usepackage{amsmath}
\usepackage{amscd}
\usepackage{amsthm}
\usepackage{amsfonts}
\usepackage{enumerate}
\usepackage{graphicx}
\usepackage{color}
\usepackage{here} 
\usepackage{bm} 
\usepackage{tikz}

\makeatletter
  
  \@addtoreset{equation}{section}
\makeatother

\newcommand{\ZZ}{\mathbb{Z}}

\newcommand{\kk}{\Bbbk}

\def\opn#1#2{\def#1{\operatorname{#2}}} 
\opn\conv{conv} \opn\mut{mut} \opn\GL{GL} \opn\cone{cone} \opn\ini{in} \opn\NF{NF} \opn\deg{deg}
\opn\NF{NF} \opn\sign{sign} \opn\mat{Mat} \opn\rank{rank} \opn\type{type} \opn\reg{reg} \opn\core{core}
\opn\indmat{im} \opn\Indeg{Indeg} \opn\star{star} \opn\link{link} \opn\Tor{Tor} \opn\MNF{MNF} \opn\Min{Min}

\newtheorem{thm}{Theorem}[section]

\newtheorem{lem}[thm]{Lemma}
\newtheorem{prop}[thm]{Proposition}

\theoremstyle{definition}

\theoremstyle{remark}
\newtheorem{rem}[thm]{Remark}

\begin{document}

\title{Inequalities of invariants on Stanley-Reisner rings of Cohen--Macaulay simplicial complexes}
\author{Akihiro Higashitani}
\author{Hiroju Kanno}
\author{Kazunori Matsuda}

\address[A. Higashitani]{Department of Pure and Applied Mathematics, Graduate School of Information Science and Technology, Osaka University, Suita, Osaka 565-0871, Japan}
\email{higashitani@ist.osaka-u.ac.jp}
\address[H. Kanno]{Department of Pure and Applied Mathematics, Graduate School of Information Science and Technology, Osaka University, Suita, Osaka 565-0871, Japan}
\email{u825139b@ecs.osaka-u.ac.jp}
\address[K. Matsuda]{Kitami Institute of Technology, Kitami, Hokkaido 090-8507, Japan}
\email{kaz-matsuda@mail.kitami-it.ac.jp}


\maketitle

\begin{abstract} 
The goal of the present paper is the study of some algebraic invariants of Stanley--Reisner rings of Cohen--Macaulay simplicial complexes of dimension $d - 1$. 
We prove that the inequality $d \leq \reg(\Delta) \cdot \type(\Delta)$ holds for any $(d-1)$-dimensional Cohen--Macaulay simplicial complex $\Delta$ satisfying $\Delta=\core(\Delta)$, 
where $\reg(\Delta)$ (resp. $\type(\Delta)$) denotes the Castelnuovo--Mumford regularity (resp. Cohen--Macaulay type) of the Stanley--Reisner ring $\kk[\Delta]$. 
Moreover, for any given integers $d,r,t$ satisfying $r,t \geq 2$ and $r \leq d \leq rt$, 
we construct a Cohen--Macaulay simplicial complex $\Delta(G)$ as an independent complex of a graph $G$ such that 
$\dim(\Delta(G))=d-1$, $\reg(\Delta(G))=r$ and $\type(\Delta(G))=t$. 
\end{abstract}

\section{Introduction}

The theory of monomial ideals is one of the most well-studied topics in the area of combinatorial commutative algebra. 
Since any monomial ideal can be deduced into a {\em squarefree} monomial ideal by polarization (see \cite[Section 1.6]{HerzogHibi}), 
Stanley--Reisner rings and Stanley--Reisner ideals play a crucial role. 
The goal of the present paper is the investigation of some algebraic invariants on Stanley--Reisner rings (or ideals). 

For the terminologies used throughout the present paper, see Section~\ref{sec:notation}. 
One typical class of Stanley--Reisner ideals is the edge ideals $I(G)$ of graphs $G$, which coincide with the Stanley--Reisner ideals of the independent complexes $\Delta(G)$ of $G$. 
Recently, the invariants on the Stanley--Reisner rings $\kk[\Delta(G)]$ of $\Delta(G)$ are intensively investigated (\cite{CN, CRT, GV, HHKO, HKKMV, HKM, HMV, KM}, and so on). 
What we would like to do in the present paper is to generalize the previous studies on the edge ideals and initiate a new study for more general squarefree monomial ideals than edge ideals.

For the investigation of the invariants on Stanley--Reisner rings of simplicial complexes $\Delta$, 
we focus on the Castelnuovo--Mumford regularity $\reg(\Delta)$ and the Cohen--Macaulay type $\type(\Delta)$. 
The first main result of the present paper is the following: 
\begin{thm}\label{main1}
Let $\Delta$ be a Cohen--Macaulay simplicial complex of dimension $d-1$ satisfying $\Delta=\core(\Delta)$. 
Then we have \begin{align}\label{ineq}d \leq \reg (\Delta) \cdot \type(\Delta).\end{align}
\end{thm}
We can also see in the following proposition that a stronger inequality holds for the simplicial complexes $\Delta$ under some stronger assumption on $\kk[\Delta]$: 
\begin{prop}\label{prop:main}
Let $\Delta$ be a Cohen--Macaulay simplicial complex of dimension $d-1$ satisfying $\Delta=\core(\Delta)$. 
Assume that $\Delta$ satisfies one of the following: 
\begin{itemize}
\item[(1)] $\kk[\Delta]$ has a $1$-linear resolution; or 
\item[(2)] $a(\Delta)=0$, where $a(\Delta)$ denotes the $a$-invariant of $\kk[\Delta]$. 
\end{itemize}
Then the inequality \begin{align}\label{ineq'} d \leq \reg(\Delta)+\type(\Delta)-1 \end{align} holds. 
\end{prop}
Note that the inequality \eqref{ineq'} implies \eqref{ineq} since $\reg(\Delta) \geq 1$ and $\type(\Delta) \geq 1$. 
In Section~\ref{sec:proofs1}, we give a proof of Theorem~\ref{main1} and Proposition~\ref{prop:main}. 

\medskip

If an inequality appears among the invariants, then it is quite natural to think of whether that is best possible or not. 
The following theorem, which is the second main result, shows that the inequality \eqref{ineq} is best possible: 
\begin{thm}\label{main2}
Let $d,r,t$ be integers with $r,t \geq 2$ and assume that the inequalities $r \leq d \leq rt$ hold. 
Then there exists a graph $G$ having no isolated vertex such that $\Delta(G)$ is Cohen--Macaulay and 
$$\dim (\Delta(G))=d-1, \;\; \reg (\Delta(G)) = r \text{ and } \type (\Delta(G))=t.$$
\end{thm}
We note that an isolated vertex of $G$ corresponds to a vertex in $V(G) \setminus \core(V(G))$, so the condition $\Delta=\core(\Delta)$ is equivalent to what $G$ has no isolated vertex. 
Moreover, the inequality $r \leq d$ naturally holds (see Remark~\ref{rem:a-inv}). 

In Section~\ref{sec:proofs2}, we give a proof of Theorem~\ref{main2}. 

\medskip

\subsection*{Acknowledgements} 
The authors would like to be grateful to Satoshi Murai for the comments on the first version of the present paper. 
Thanks to his comments, the authors could improve Theorem~\ref{main1} of the first version and make the present paper more readable. 

The first named author is partially supported by JSPS Grant-in-Aid for Scientific Research (C) 20K03513. 
The third named author is partially supported by JSPS Grant-in-Aid for Scientific Research (C) 20K03550.

\bigskip


\section{Notation and terminologies}\label{sec:notation}

In this section, we collect the notation used in the present paper. 
Please consult, e.g., \cite[Section 5]{BH} or \cite{HerzogHibi}, for the detailed information. 

Let $V=\{v_1,\ldots,v_n\}$ be a finite set. 
We call a family $\Delta$ of subsets of $V$ an {\em (abstract) simplicial complex} on the vertex set $V$ 
if it satisfies that \begin{itemize}
\item[(i)] $\{v_i\} \in \Delta$ for any $v_i \in V$; and 
\item[(ii)] if $F \in \Delta$ and $F' \subset F$, then $F' \in \Delta$. 
\end{itemize}
Note that $\emptyset$ belongs to $\Delta$ for any simplicial complex $\Delta$. 
Let $\dim (\Delta)$ denote the dimension of $\Delta$, i.e., $\dim (\Delta)=\max\{\dim (F) : F \in \Delta\}$. 

Throughout this section, let $\Delta$ be a simplicial complex of dimension $d-1$ on the vertex $V=\{v_1,\ldots,v_n\}$. 

\subsection{Terminologies on simplicial complexes}

Given $F \in \Delta$, let 
\begin{align*}
&\Delta \setminus F =\{G \in \Delta : G \subset V \setminus F\}, \\
&\star(F)=\{G \in \Delta : F \cup G \in \Delta\}, \text{ and }\\
&\link(F)=\star(F) \setminus F. 
\end{align*}
When $F=\{v\}$, we use the notation $\Delta \setminus v$ instead of $\Delta \setminus \{v\}$, and so on. 

Given a new vertex $v$, let $\Delta * v$ be a new simplicial complex on the vertex set $V \sqcup \{v\}$ consisting of $\{F, F \cup \{v\} : F \in \Delta\}$. 
We call $\Delta * v$ a {\em cone} of $\Delta$. Note that $\star_{\Delta * v}(v)=\Delta * v$. 

For $W \subset V$, let $\Delta_W=\{F \in \Delta : F \subset W\}$. 
Let $\core(V)=\{ v \in V : \star(v) \neq \Delta\}$ and let $\core(\Delta)=\Delta_{\core(V)}$.  
In the sequel, we often assume that $\Delta=\core(\Delta)$, which is equivalent to $\core(V)=V$.

We recall the definition of vertex decomposable simplicial complexes. 
A simplicial complex $\Delta$ on the vertex set $V$ is called {\em vertex decomposable} (\cite{BW}) if $\Delta$ is a simplex 
or if there is a vertex $v \in V$, called a {\em shedding vertex}, such that 
\begin{itemize}
\item $\Delta \setminus v$ and $\link(v)$ are vertex decomposable; and 
\item no face of $\link(v)$ is a facet of $\Delta \setminus v$. 
\end{itemize}
It is well known that for a pure simplicial complex $\Delta$, the vertex decomposability of $\Delta$ implies Cohen--Macaulayness of $\Delta$. 
(Here, we say that $\Delta$ is {\em pure} if all facets of $\Delta$ have the same dimension.)

Let $\Delta_1,\Delta_2$ be two simplicial complexes on disjoint vertex sets of dimension $d-1$. 
Let $R_1 \in \Delta_1$ and $R_2 \in \Delta_2$ be ridges, which are faces of dimension $d-2$, respectively.
We define a new simplicial complex, 
called the {\em ridge sum of $\Delta_1$ and $\Delta_2$}, 
by identifying $R_1$ and $R_2$ in the union $\Delta_1 \cup \Delta_2$ of simplicial complexes. 
The terminology of ridge sums was introduced in \cite{MM} and similar notions appear in \cite{DS}. 

We say that a $(d-1)$-dimensional pure simplicial complex $\Delta$ is {\em strongly connected} if for any two facets $F,F'$ of $\Delta$, 
there exists a sequence of facets $F_0,F_1,\ldots,F_p$ of $\Delta$ such that $F_0=F$, $F_p=F'$ and $\dim (F_{i-1} \cap F_i) = d-2$ for each $1 \leq i \leq p$.

\subsection{Terminologies on Stanley--Reisner rings}
We refer the reader to \cite{BH} for the introduction to the theory of Stanley--Reisner rings. 

Let $S=\kk[x_1,\ldots,x_n]$ be the polynomial ring with $n(=|V|)$ variables over a field $\kk$. 
We denote by $I_\Delta$ (resp. $\kk[\Delta]$) the {\em Stanley--Reisner ideal} (resp. {\em Stanley--Reisner ring}) of $\Delta$ 
over a field $\kk$, i.e., 
\begin{align*}
I_\Delta=\left(\prod_{j=1}^\ell x_{i_j} : \{v_{i_1},\ldots,v_{i_\ell}\} \not\in \Delta \right) \subset S \text{ and }
\kk[\Delta]=S/I_\Delta. 
\end{align*}

The invariants of $\Delta$ or $\kk[\Delta]$ or $I_\Delta$ discussed in the present paper are listed below: 
\begin{itemize}
\item Let $\dim(\kk[\Delta])$ denote the Krull dimension of $\kk[\Delta]$. Then $\dim(\kk[\Delta])=\dim(\Delta)+1$. 
\item We say that $\Delta$ is Cohen--Macaulay over $\kk$ if $\kk[\Delta]$ is Cohen--Macaulay. 
In the sequel, we often omit ``over $\kk$''. 
By Reisner's criterion (\cite[Corollary 5.3.9]{BH}), $\Delta$ is Cohen--Macaulay if and only if 
$\widetilde{H}_i(\link(F))=0$ for all $F \in \Delta$ and $i < \dim(\link(F))$, 
where $\widetilde{H}_i(\Gamma)$ denotes the $i$-th reduced simplicial homology of a simplicial complex $\Gamma$ with values in $\kk$. 
\item Let $\omega_\Delta$ denote the canonical module of $\kk[\Delta]$. 
\item On the Tor modules, we use the notation $\Tor_i^S(\Delta)$ instead of $\Tor_i^S(\kk,\kk[\Delta])$. 
\item Let $\beta_{ij}(\Delta)=\dim_\kk (\Tor_i^S(\Delta)_{i+j})$ for $i,j \in \ZZ$, where $\dim_\kk$ stands for the dimension as a $\kk$-vector space, 
and let $\beta_i(\Delta)=\sum_{j \in \ZZ}\beta_{ij}(\Delta)$. We call $\beta_i(\Delta)$ the {\em $i$-th graded Betti number} of $\kk[\Delta]$. 
Note that by Auslander--Buchsbaum formula (see \cite[Theorem 1.3.3]{BH}), we see that $\Delta$ is Cohen--Macaulay if and only if the projective dimension of $\kk[\Delta]$ is $n-d$, 
which is equivalent to $\beta_i(\Delta)=0$ for any $i > n-d$ and $\beta_{n-d}(\Delta) \neq 0$ (see \cite[Corollary 1.3.2]{BH}). 
\item Let $\reg(\kk[\Delta])$ denote the {\em Castelnuovo--Mumford regularity} of $\kk[\Delta]$, i.e., 
$\reg(\kk[\Delta])=\max\{j-i : \beta_{ij}(\Delta) \neq 0\}$. We use the notation $\reg(\Delta)$ instead of $\reg(\kk[\Delta])$. 
\item Assume that $\Delta$ is Cohen--Macaulay over $\kk$. 
Let $\type(\kk[\Delta])$ denote the {\em Cohen--Macaulay type} of $\kk[\Delta]$. Namely, $\type(\kk[\Delta])=\beta_c(\Delta)$, where $c=n-\dim(\kk[\Delta])$. 
Similarly to the above, we use the notation $\type(\Delta)$ instead of $\type(\kk[\Delta])$. 
\item Moreover, when $\Delta$ is Cohen--Macaulay over $\kk$, let $a(\kk[\Delta])=a(\Delta)$ denote the {\em $a$-invariant} (see \cite{GW}) of $\kk[\Delta]$. 
Namely, we have $a(\Delta)=-\min\{ j : (\omega_\Delta)_j \neq 0\}$. 
It is well known that $$\reg(\Delta)=a(\Delta)+\dim(\kk[\Delta]).$$ 
\item Let $\Indeg(I_\Delta)$ denote the maximal degree of the minimal system of generators of $I_\Delta$. 
Let $$\MNF(\Delta)=\{G \subset V : G \not\in \Delta, \; G' \in \Delta \text{ for any }G' \subsetneq G\}.$$ 
Namely, $\MNF(\Delta)$ is the set of {\em minimal non-faces} of $\Delta$. Note that $\Indeg(I_\Delta)=\max\{|G| : G \in \MNF(\Delta)\}$. 
\item We say that $\kk[\Delta]$ has a {\em $p$-linear resolution} if $\beta_{ij}(\Delta) = 0$ unless $j-i=p$. 
In particular, $\reg(\Delta)=p$ in this case. 
Note that if $\kk[\Delta]$ has a $p$-linear resolution, then $\Indeg(I_\Delta)=p+1$, but the converse is not true in general. 
\end{itemize}
\begin{rem}\label{rem:a-inv}
It is known that $a(\Delta) = \reg(\Delta) - \dim(\kk[\Delta]) \leq 0$ (see \cite[Corollary B.4.1]{V}). 
Thus, the inequality $\reg(\Delta) \leq \dim (\Delta)+1$ always holds for any Cohen--Macaulay simplicial complex $\Delta$. 
\end{rem}
\begin{rem}\label{strongly_connected}
Any Cohen--Macaulay simplicial complex $\Delta$ is always strongly connected. 
In fact, if a $(d-1)$-dimensional simplicial complex $\Delta$ is not strongly connected, then there should exist two facets $F,F'$ such that 
$\dim (F \cap F') < d-2$ and $\link(F \cap F')$ is disconnected. When $\Delta$ is pure, we have $\dim (\link(F \cap F'))=d-2-\dim (F \cap F') > 0$. 
Hence, $\link (F \cap F')$ is not Cohen--Macaulay by Reisner's criterion (see also \cite[Exercise 5.1.26]{BH}). 
On the other hand, $\link(G)$ must be Cohen--Macaulay for any $G \in \Delta$ if $\Delta$ is Cohen--Macaulay. 
\end{rem}

\subsection{Terminologies on graphs}

Throughout the present paper, we only treat finite simple graphs (i.e., a finite graph with no loops and no multiple edges), and we omit ``finite simple''. 

Let $G$ be a graph on the vertex set $V(G)$ with the edge set $E(G)$. 
We call a subset $W$ of $V(G)$ an {\em independent set} if $\{v,w\} \not\in E(G)$ for any $v,w \in W$. 
Then we can associate an abstract simplicial complex $\Delta(G)$ on $V(G)$ as follows: 
$$\Delta(G)=\{ W \subset V(G) : W \text{ is an independent set of }G\},$$
which is called an {\em independent complex} of $G$. 

\begin{itemize}
\item For a vertex $v \in V(G)$, let $N_G(v)=\{w \in V(G) : \{v,w\} \in E(G)\}$ and let $N_G[v]=N_G(v) \cup \{v\}$. 
\item For a subset $S \subset V(G)$, we denote by $G|_S$ the subgraph of $G$ on the vertex set $S$ with the edge set $\{\{v,w\} : v,w \in S\} \cap E(G)$. 
We use the notation $G \setminus S$ instead of $G|_{V(G) \setminus S}$. 
\item Let $M \subset E(G)$ be a subset of the edge set. We say that $M$ is an {\em induced matching} of $G$ if $M$ satisfies that 
\begin{itemize}
\item $e \cap e' = \emptyset$ for any $e,e' \in M$ with $e \neq e'$; and 
\item there is no $e'' \in E(G)$ such that $e \cap e'' \neq \emptyset$ and $e' \cap e'' \neq \emptyset$. 
\end{itemize}
Moreover, let $\indmat(G)$ denote the maximal cardinality among induced matchings of $G$, called the {\em induced matching number} of $G$. 
\item For a subset $C \subset V(G)$, we say that $C$ is a {\em vertex cover} of $G$ if $C \cap e \neq \emptyset$ for any $e \in E(G)$. 
\end{itemize}

Let $V(G)=\{v_1,\ldots,v_n\}$ and consider the polynomial ring $S=\kk[x_1,\ldots,x_n]$. 
The {\em edge ideal} $I(G)$ of a graph $G$ is defined by $I(G)=( x_ix_j : \{v_i,v_j\} \in E(G)) \subset S$. 
Edge ideals are very well studied around the area of combinatorial commutative algebra. 
We notice that the edge ideal of $G$ coincides with the Stanley--Reisner ideal of $\Delta(G)$.

\bigskip


\section{Proofs of Theorem~\ref{main1} and Proposition~\ref{prop:main}}\label{sec:proofs1}

The goal of this section is to give proofs of Theorem~\ref{main1} and Proposition~\ref{prop:main}.

Let $\Delta$ be a Cohen--Macaulay simplicial complex of dimension $d-1$ satisfying $\Delta=\core(\Delta)$. Let 
\begin{align*}X(\Delta)=\{ F : \widetilde{H}_{\dim(\link(F))}(\link(F))\neq 0\}. \end{align*}
We define $M(\Delta)$ by setting the set of minimal faces of $X(\Delta)$ with respect to inclusion. 

We claim the following lemma: 
\begin{lem}\label{lem:key}
Work with the same notation as above. Then we have the following: 
\begin{align*}
&(1) \quad \bigcup_{F \in M(\Delta)} \star(F)=\Delta; \\
&(2) \quad d-|F| \leq \reg(\Delta) \text{ for each }F \in M(\Delta); \\
&(3) \quad |M(\Delta)| \leq \type(\Delta). 
\end{align*}
\end{lem}
\begin{proof}
(1) It is enough to show that $H \in \bigcup_{F \in M(\Delta)} \star(F)$ for each facet $H$ of $\Delta$. 
Given any facet $H \in \Delta$, since $\link(H)=\{\emptyset\}$ satisfies that $\dim_\kk(\widetilde{H}_{-1}(\link(H)) = 1$, 
we see that $H \in X(\Delta)$. Therefore, $H \in \star(F)$ for some $F \in M(\Delta)$. 

\noindent
(2) By \cite[Exercise 5.6.6]{BH}, we see that 
\begin{align}\label{non-vanishing}(\omega_\Delta)_j \neq 0 \;\Longrightarrow\; j \geq |F| \text{ for }F \in X(\Delta).\end{align} 
Hence, for each $F \in M(\Delta)$. we have $$-a(\Delta)=\min\{ j : (\omega_\Delta)_j \neq 0\} \leq |F|.$$ 
Therefore, we obtain that $\reg(\Delta)=a(\Delta)+d \geq d-|F|$. 

\noindent
(3) We see from \eqref{non-vanishing} that the elements of $X(\Delta)$ correspond to the non-vanishing squarefree $\ZZ^n$-degree components of $\omega_\Delta$, 
where $n$ is the number of vertices of $\Delta$. 
Since $\type(\Delta)$ is equal to the number of minimal generators of $\omega_\Delta$ and each $F$ in $M(\Delta)$ should correspond to a minimal generator of $\omega_\Delta$, we obtain that $|M(\Delta)| \leq \type(\Delta)$. 
\end{proof}

\begin{proof}[Proof of Theorem~\ref{main1}]
By our assumption $\Delta=\core(\Delta)$, there are facets $H,H'$ of $\Delta$ with $H \cap H' = \emptyset$. 
Moreover, the Cohen--Macaulayness of $\Delta$ implies that $\Delta$ is strongly connected (see Remark~\ref{strongly_connected}). 
Hence, there exists a sequence of facets $F_1,\ldots,F_p$ such that 
$F_1=H$, $F_p=H'$ and $\dim (F_i \cap F_{i+1}) = d-2$ for each $1 \leq i \leq p-1$. 
By Lemma~\ref{lem:key} (1), there exists a sequence of elements $S_1,\ldots,S_q \in M(\Delta)$ such that 
$F_{i_{j-1}+1},\ldots,F_{i_j} \in \star(S_j)$ for each $1 \leq j \leq q$, where $1=i_0 < i_1 < \cdots < i_q=p$. 
Then we may assume that 
\begin{align}\label{we_may_assume}
\text{$\star(S_i)$ and $\star(S_{i+1})$ contain a common facet for each $i$.}
\end{align}
In fact, by choice of a sequence $S_1,\ldots,S_q$, $\star(S_i)$ and $\star(S_{i+1})$ contain a common face $H$ with $\dim (H)=d-2$. 
Since there are at least two facets containing $H$, i.e., $\link(H)$ consists of at least two vertices, we have $H \in X(\Delta)$. 
Hence, there is $T \in M(\Delta)$ with $T \subset H$. If either $T=S_i$ or $T=S_{i+1}$ holds, 
then we see that $\star(S_i)$ and $\star(S_{i+1})$ contain a common facet. Even if $T \neq S_i$ and $T \neq S_{i+1}$, 
we may add $T$ between $S_i$ and $S_{i+1}$, i.e., we replace the sequence by $S_1,\ldots,S_i,T,S_{i+1},\ldots,S_q$. 
Then $\star(S_i)$ and $\star(T)$ contain a common facet, and so do $\star(T)$ and $\star(S_{i+1})$. 

Since we may also assume that $S_1,\ldots,S_q$ are all distinct, 
we can take a required sequence $S_1,\ldots,S_q$ satisfying \eqref{we_may_assume}. 
Then we can see that \begin{align}\label{inc_exc}
d-\left| \bigcap_{i=1}^j S_i \right| \leq \sum_{i=1}^j (d-|S_i|) \text{ for each }1 \leq j \leq q 
\end{align}
by induction. The case $j=1$ is trivial. Let $j>1$ and assume that \eqref{inc_exc} holds for $j-1$. 
Take a common facet $G$ of $\star(S_{j-1})$ and $\star(S_j)$. 
Then we have $|(\bigcap_{i=1}^{j-1}S_i) \cup S_j| \leq |S_{j-1} \cup S_j| \leq |G|=d$. 
Thus, by the hypothesis of induction, we conclude that 
\begin{align*}
d-\left| \left(\bigcap_{i=1}^{j-1} S_i\right) \cap S_j\right| 
&= d - \left(\left| \bigcap_{i=1}^{j-1} S_i\right| + |S_j| - \left| \left(\bigcap_{i=1}^{j-1} S_i\right) \cup S_j\right|\right) \\
&\leq \sum_{i=1}^{j-1}(d-|S_i|) -|S_j| + d = \sum_{i=1}^j(d-|S_i|). 
\end{align*}

Here, we have that $S_1 \subset H$ and $S_q \subset H'$. Hence, $S_1 \cap S_q \subset H \cap H' = \emptyset$. 
In particular, $\bigcap_{i=1}^q S_i= \emptyset$ holds. Therefore, by applying \eqref{inc_exc} with $j=q$, we obtain that 
\begin{align*}
d&=d-\left| \bigcap_{i=1}^q S_i \right| \leq \sum_{i=1}^q (d-|S_i|)\leq \max\{d-|S_i| : 1 \leq i \leq q\} \cdot |M(\Delta)| \\
&\leq \reg(\Delta) \cdot \type(\Delta) 
\end{align*}
by Lemma~\ref{lem:key} (2) and (3), as desired. 
\end{proof}

\medskip

Next, we prove Proposition~\ref{prop:main}. 
Here, we know that for a simplicial complex $\Delta$, $\Indeg(I_\Delta)=2$ holds if and only if there is a graph $G$ such that $\Delta=\Delta(G)$. 
Since $\Indeg(I)=2$ holds if $S/I$ has a $1$-linear resolution and $I$ contains no linear polynomial, 
we may discuss the edge ideals of graphs in the case (1). 

The following theorem is used for the proof of Proposition~\ref{prop:main} (1). 
\begin{thm}[{\cite[Theorem 3.3]{GV}}]\label{dim_codim}
Let $G$ be a graph with $n$ vertices all of which are non-isolated, and let $d-1=\dim(\Delta(G))$. 
Suppose that $\Delta(G)$ is Cohen--Macaulay. Then we have $n-d \geq d$. 
\end{thm}

\begin{proof}[Proof of Proposition~\ref{prop:main}]
(1) Let $n$ be the number of vertices of $G$ and let $c=n-d$, where $d-1=\dim(\Delta(G))$. 
Here, Betti numbers of homogeneous Cohen--Macaulay $\kk$-algebras with $p$-linear resolutions over the polynomial ring 
are completely determined by Eisenbud--Goto (\cite[Proposition 1.7 (c)]{EG}) and 
we know that $\type(\Delta)=\beta_c(\Delta) = \binom{c+p-1}{p}$ for a Cohen--Macaulay simplicial complex 
whose Stanley--Reisner ring has a $p$-linear resolution. From this, we obtain that $\type(\Delta(G))=c$. 
Since $\reg(\Delta(G))=1$, it follows from Proposition~\ref{dim_codim} that $\reg(\Delta(G))+\type(\Delta(G))-1 =c \geq d$, as required. 

\noindent
(2) Let $\Delta$ be a Cohen--Macaulay simplicial complex of dimension $d-1$ satisfying $\Delta=\core(\Delta)$. 
If $a(\Delta)=0$, since $a(\Delta)=\reg(\Delta)-\dim (\kk[\Delta])$, the inequality \eqref{ineq'} trivially holds. 
Note that if $\Delta$ is Gorenstein and $\Delta=\core(\Delta)$, then $a(\Delta)=0$ (see \cite[Exercise 5.6.8]{BH}). 
Thus, the inequality \eqref{ineq'} holds, as required. 
\end{proof}

\bigskip

\section{Proof of Theorem~\ref{main2}}\label{sec:proofs2}

The goal of this section is to construct examples of simplicial complexes $\Delta$ such that 
the triple $(\dim (\Delta), \reg (\Delta), \type (\Delta))$ satisfies the required inequalities. 
We construct such simplicial complexes as independent complexes. 

Before the construction of examples, 
let us recall two constructions of graphs from a given graph $G$; {\em whiskered graph $W(G)$} and {\em $S$-suspension $G^S$}. 
Let $G$ be a graph on the vertex set $V(G)=\{v_1,\ldots,v_d\}$ with the edge set $E(G)$. 

\subsection{Whiskered graphs}\label{sec:W(G)}

We define the {\em whiskered graph} of $G$, denoted by $W(G)$, by setting 
\begin{equation}\label{notation:W(G)}
\begin{split}
V(W(G))=V(G) \sqcup \{w_1,\ldots,w_d\} \;\text{ and }\; E(W(G))=E(G) \sqcup \{\{v_i,w_i\} : i=1,\ldots,d\}, 
\end{split}
\end{equation}
where $w_1,\ldots,w_d$ are new vertices. 

\begin{lem}[{cf. \cite{CN}, \cite[Theorem 1.1]{HHKO}}]\label{lem:whisker}
For any graph $G$ with $d$ vertices, $\Delta(W(G))$ is pure and vertex decomposable. 
In particular, $\Delta(W(G))$ is a Cohen--Macaulay simplicial complex of dimension $d-1$. 
\end{lem}

%

\subsection{$S$-suspensions}

Given an independent set $S$ of $G$, we define the {\em $S$-suspension} of $G$, denoted by $G^S$, by setting 
$$V(G^S)=V(G) \sqcup \{w\}\text{ and }E(G^S)=E(G) \sqcup \{\{v,w\} : v \not\in S\},$$
where $w$ is a new vertex. 
Note that $S \sqcup\{w\}$ becomes an independent set of $G^S$. Hence, in the language of the independent complex, 
$\Delta(G^S)$ is the ridge sum of $\Delta(G)$ and a new simplex $S \sqcup \{w\}$ along the ridge $S$. 

\begin{lem}[{\cite[Lemma 1.2]{HKKMV}, \cite[Theorem 2.9]{DS}, \cite[Lemma 1.5]{HKM}}]\label{lem:S-suspension}
Let $G$ be a graph and let $S$ be an independent set of $G$. 
\begin{enumerate}
\item Assume that $\Delta(G)$ is Cohen--Macaulay. If $|S|=\dim(\Delta(G))$, then $\Delta(G^S)$ is also Cohen--Macaulay of the same dimension as $\Delta(G)$. 
\item Assume that $|S|=\dim(\Delta(G))$. Then $\type(\Delta(G^S))=\type(\Delta(G))+1$.
\item Assume that $G$ has no isolated vertex. Then $\reg(\Delta(G^S))=\reg(\Delta(G))$. 
\end{enumerate}
\end{lem}
\begin{proof}
The statement (1) (resp. (3)) is a direct consequence of \cite[Lemma 1.2]{HKKMV} (resp. \cite[Lemma 1.5]{HKM}). 
The statement (2) follows from \cite[Theorem 2.9]{DS}. 
\end{proof}

\subsection{Whiskered graphs of complete multi-partite graphs}

For the construction of our desired graph, we consider the whiskered graphs of complete multi-partite graphs, 
which play the essential role in the proof of Theorem~\ref{main2}. 
\begin{prop}\label{complete_multipartite}
Let $r_1,\ldots,r_t$ be integers with $r_1 \geq \cdots \geq r_t \geq 1$. Consider $G=W(K_{r_1,\ldots,r_t})$. Then 
$$\dim(\Delta(G)) = \sum_{i=1}^t r_i-1,\;\; \reg(\Delta(G))=r_1 \;\text{ and }\; \type(\Delta(G))=t.$$
\end{prop}

For the proof of Proposition~\ref{complete_multipartite}, we recall a result from \cite{CRT}. 
Let $G$ be a graph on the vertex set $V(G)$ with $2d$ vertices all of which are non-isolated vertices. We consider the condition that 
\begin{equation}\label{condition}
\begin{split}
&V(G) = X \sqcup Y, \;\text{where }\\
&\text{$X=\{x_1,\ldots,x_d\}$ is a minimal vertex cover of $G$, and } \\
&\text{$Y=\{y_1,\ldots,y_d\}$ is a maximal independent set of $G$, such that }\\
&\{\{x_i,y_i\} : i=1,\ldots,n\} \subset E(G). 
\end{split}
\end{equation}

Consider the graph satisfying \eqref{condition}. Let $E_i=\{ k \in \{1,\ldots,d\} : \{x_k,y_i\} \in E(G)\} \setminus \{i\}$ for $i=1,\ldots,d$, 
and let $O_{[d]}(G)$ be the graph on the vertex set $V(O_{[d]}(G))=V(G)$ with the edge set 
\begin{align*}
E(O_{[d]}(G))=\left(E(G) \setminus \bigcup_{i=1}^d\{\{x_k,y_i\} : k \in E_i\}\right) \cup \bigcup_{i=1}^d\{\{x_k,x_i\} : k \in E_i\}. 
\end{align*}
\begin{prop}[{\cite[Corollary 4.4]{CRT}}]\label{prop:type}
Let $G$ be a graph with $2d$ vertices all of which are non-isolated vertices and let $\dim(\Delta(G))=d-1$. 
Assume that $\Delta(G)$ is Cohen--Macaulay and $G$ satisfies \eqref{condition}. Then 
$$\type(\Delta(G))=\upsilon(O_{[n]}(G)|_X),$$
where $\upsilon(H)$ denotes the number of minimal vertex covers of a graph $H$. 
\end{prop}

Before giving the proof of Proposition~\ref{complete_multipartite}, we fix the notation on $K_{r_1,\ldots,r_t}$. 
Let $V_i=\{x_1^{(i)},\ldots,x_{r_i}^{(i)}\}$ for $i=1,\ldots,t$ and let $V(K_{r_1,\ldots,r_t})=\bigsqcup_{i=1}^t V_i$. 
Let $G=W(K_{r_1,\ldots,r_t})$, let $V(G)=\{x^{(i)}_j, y^{(i)}_j : 1 \leq i \leq t, 1 \leq j \leq r_i\}$, and let 
\begin{align*}
E(G)=E(K_{r_1,\ldots,r_t}) \sqcup \{\{x_j^{(i)},y_j^{(i)}\} : 1 \leq i \leq t, 1 \leq j \leq r_i\}.
\end{align*}
\begin{proof}[Proof of Proposition~\ref{complete_multipartite}]
First, $\dim(\Delta(G))=\sum_{i=1}^tr_i-1$ directly follows from Lemma~\ref{lem:whisker}. 

\medskip

Next, we show that $\reg(\Delta(G))=r_1$. 
Here, it follows from \cite[Theorem 2.4]{KM} that $\reg(\Delta(H))=\indmat(H)$ 
if a graph $H$ contains no induced cycle of length $5$ and $\Delta(H)$ is vertex decomposable. 
By the structure of $G$, we see that $G$ contains no induced cycle of length $5$. 
Moreover, Lemma~\ref{lem:whisker} implies that $\Delta(G)$ is vertex decomposable. 
Since we see that $\{\{x_j^{(i)},y_j^{(i)}\} : j=1,\ldots,r_1\}$ forms an induced matching of $G$ for each $i$ and those are maximal ones, 
we conclude that $\indmat(G)=r_1$. 

\medskip

Finally, we show that $\type(\Delta(G))=t$. Here, we can check that $G$ satisfies \eqref{condition} 
by setting $X=V(K_{r_1,\ldots,r_t})$ and $Y=V(G) \setminus X$. Under this setting, we see that $O_{[d]}(G)=E(K_{r_1,\ldots,r_t})$, 
i.e., all vertices in $Y$ become isolated in $O_{[d]}(G)$. 
Hence, we may count the number of minimal vertex covers of $K_{r_1,\ldots,r_t}$. Let $V=V(K_{r_1,\ldots,r_t})$. 
Then $C \subset V$ is a minimal vertex cover if and only if $C=V \setminus V_i$ for some $i$. 
In fact, if there are $i$ and $i'$ with $i \neq i'$ such that $V_i \setminus C \neq \emptyset$ and $V_{i'} \setminus C \neq \emptyset$, 
since we consider a complete multi-partite graph, there must be an edge between $V_i \setminus C$ and $V_{i'} \setminus C$, a contradiction. 
Therefore, one has $\upsilon(O_{[d]}(G|_X))=\upsilon(O_{[d]}(K_{r_1,\ldots,r_t}))=t$, so we conclude that $\type(\Delta(G))=t$ by Proposition~\ref{prop:type}, as required. 
\end{proof}

Now, we are ready to prove Theorem~\ref{main2}. 
\begin{proof}[Proof of Theorem~\ref{main2}]
Let $d,r,t$ be integers with $r,t \geq 2$ and $r \leq d \leq rt$. 

Write $d=pr+q$, where $0 \leq q < r$, i.e., $p$ (resp. $q$) is the quotient (resp. the remainder) of $d$ divided by $r$. 
Note that we have $1 \leq p \leq t$ and $q=0$ if $p=t$ by our assumption. 

Let $G=W(K_{\underbrace{r,\ldots,r}_p,q})$ if $q \neq 0$ and let $G=W(K_{\underbrace{r,\ldots,r}_p})$ if $q=0$. 
Note that $|V(K_{\underbrace{r,\ldots,r}_p,q})|=pr+q=d$. 
Thus, it follows from Lemma~\ref{lem:whisker} that $\Delta(G)$ is Cohen--Macaulay of dimension $d-1$. 
Moreover, by Proposition~\ref{complete_multipartite}, we know that $$\reg(\Delta(G))=r \text{ and }\type(\Delta(H))=p+1.$$ 
If $p+1 < t$, by applying certain $S$-suspensions of $(t-p-1)$ times, we obtain the desired graph by Lemma~\ref{lem:S-suspension}. 
\end{proof}

\end{document}